\documentclass[letterpaper, 10 pt, conference]{ieeeconf}  

\IEEEoverridecommandlockouts                              
\usepackage{graphics}
\usepackage{graphicx}
\usepackage{epsfig,amssymb}
\usepackage{amsmath}
\usepackage{mathrsfs}
\usepackage{color}
\usepackage{array}
\usepackage{amsfonts,booktabs}
\usepackage{lipsum,amsmath,multicol}
\usepackage{graphicx}
\usepackage{subfigure}
\usepackage{times}
\usepackage{pdfsync}
\usepackage{pgfplots}
\usepackage{pgfplotstable}
\usepackage[subnum]{cases}
\usepackage{filecontents}
\usepackage{multirow}
\usepackage{algorithm}
\usepackage{algpseudocode}
\usepackage{float}
\usepackage{setspace}
\newtheorem{theorem}{Theorem}
\newtheorem{lemma}{Lemma}
\newtheorem{definition}{Definition}
\newtheorem{corollary}{Corollary}
\newtheorem{proposition}{Proposition}

\newtheorem{remark}{Remark}

\newcommand{\beq}{\begin{equation}}
\newcommand{\eeq}{\end{equation}}
\newcommand{\beqa}{\begin{eqnarray}}
\newcommand{\eeqa}{\end{eqnarray}}

\newcommand{\paren}[1]{\left(#1\right)}

\newcommand{\field}[1]{\ensuremath{\mathbb{#1}}}
\newcommand{\abs}[1]{\left|#1\right|} 
\newcommand{\R}{\ensuremath{\field{R}}} 
\newcommand{\1}{\ensuremath{\mathbf{1}}} 
\newcommand{\I}[1]{\ensuremath{\mathsf{1}_{\left\{#1\right\}}}} 
\newcommand{\ra}{\ensuremath{\rightarrow}} 
\newcommand{\PR}[1]{\ensuremath{\mathsf{Pr}\left\{#1\right\}}} 
\newcommand{\PRP}[1]{\ensuremath{\mathsf{Pr}\left(#1\right)}} 
\newcommand{\PRPI}[2]{\ensuremath{\mathsf{Pr}_{#1}\left(#2\right)}} 
\newcommand{\ESI}[2]{\ensuremath{\mathsf{E}_{#1}\left[#2 \right]}} 

\newcommand{\BO}[1]{\ensuremath{O\paren{#1}}}

\newcommand{\norm}[2]{\left\Vert#1\right\Vert_{#2}}

\newcommand{\utility}{u\left(x_{i},z_{i,n},\alpha_i\right)}

\usepackage[active]{srcltx} 

\usepackage[normalem]{ulem}

\title{\LARGE \bf Nash Equilibrium Approximation under Communication and Computation Constraints in
Large-Scale Non-cooperative Games}

\author{Ehsan Nekouei, Tansu Alpcan and Girish Nair
\thanks{Department of Electrical and Electronic Engineering, The University of Melbourne, VIC 3010, Australia. E-mails: \{ehsan.nekouei, tansu.alpcan, gnair\}@unimelb.edu.au }}

\begin{document}

\maketitle
\thispagestyle{empty}
\pagestyle{empty}

\begin{abstract}
This paper studies the problem of Nash equilibrium approximation in large-scale heterogeneous (static) mean-field games under communication and computation constraints. A deterministic mean-field game is considered in which the utility function of each agent depends on its action, the average of other agents' actions (called the mean variable of that agent) and a deterministic parameter. It is shown that the equilibrium mean variables of all agents converge uniformly to a constant, called asymptotic equilibrium mean (AEM), as the number of agents tends to infinity. The AEM, which depends on the limit of empirical distribution of agents' parameters, determines the asymptotic equilibrium behavior of agents. Next, the problem of approximating the AEM at a processing center under communication and computation constraints is studied. Three approximation methods are proposed to substantially reduce the communication and computation costs of approximating AEM at the processing center. The accuracy of the proposed approximation methods is analyzed and illustrated through numerical examples.
\end{abstract}

\section{INTRODUCTION}
Game theory provides mathematical frameworks for analyzing and predicting the strategic behavior of profit-maximizing agents. Computing the solution of a game is a central  research problem in the game theory literature. However, in large population games, it may not be computationally feasible to precisely calculate a game's solution. Furthermore, even when the computation of solution is possible, it may not provide any insight on the interplay between the equilibrium and population of the game. To understand the equilibrium behavior of large-scale games, which are computationally difficult to solve, it is common to analyze them in the infinite population limit. Despite the asymptotic nature of this method, it provides approximate solutions for large-scale games and sheds light on their equilibrium-population interplay.

Finding an approximate solution in large-scale \emph{heterogeneous} games, where the agents' utilities are asymmetric, becomes a cumbersome task as such solutions depend on a large number of parameters. When the approximate solution is computed centrally, \emph{e.g.,} using the cloud computing technology, each agent needs to transmit its parameters to the ``cloud" which requires  substantial communication and storage resources. The current paper studies the problem of approximating the equilibrium behavior of large-scale games under communication and computation constraints.
\subsection{Related work}
Aumann in \cite{Aumann64} showed that in a market with a continuum of traders the core of the market, \emph{i.e.,} the set of allocations in which traders have no  motivation for coalition, coincides with the equilibrium allocation, thus, the market becomes perfectly competitive. The authors in \cite{MS78} studied the value of a voting game with large number of agents. These results were extended to cooperative games in \cite{Neyman02}. Haurie and Marcotte in \cite{HM85} considered a transportation network with $m$ source-destination pairs in which each source-destination pair has $n$ \emph{identical} agents sharing the same link traversal cost and the same demand function. They showed that as $n$ becomes large the \emph{equilibrium flow} of links is characterized by a Wardop equilibrium. We note that the existence of pure Nash equilibrium for the non-cooperative games with a continuum of players, \emph{i.e,} with an uncountable set of agents, has been studied in \cite{YZ05,Rauh03,KS02} and \cite{KRS97}.


\emph{Homogeneous deterministic} mean-field games, wherein each agent's utility depends on the mean of all agents' states (or control strategies), have been investigated in the literature.  The papers \cite{BP13} and \cite{BB14} examined the equilibrium behavior of  homogenous continuous-time (deterministic) mean-field games as well as their engineering applications. The authors in \cite{MCH13} proposed a mean-field based decentralized equilibrium seeking algorithm for a homogenous discrete-time game with application in the plug-in electric vehicle charging problem. The paper \cite{GPCL16} proposed decentralized equilibrium seeking methodologies for linear quadratic mean-field games with convex states. The interested reader is refereed to \cite{HMC06, Huang10, NC13} and references therein for a thorough treatment of the stochastic mean filed games with minor and major agents.

Different from these works, the current paper focuses on the problem of approximating the equilibrium strategies of agents, in heterogeneous,  static (deterministic) mean-field games, under communication and computation constraints.
 	
\subsection{Contributions}
This paper considers the problem of approximating the Nash equilibrium of a non-cooperative mean-field game at a processing center. A heterogeneous, deterministic mean-field game is considered in which  the (non-linear) utility function of each agent $i$ depends on its action, the average of other agents' actions (called the mean variable of agent $i$) and a \emph{deterministic} parameter $\alpha_i$.  First, the existence and uniqueness of the Nash equilibrium (NE) of the considered game are studied. It is shown that the equilibrium mean variables of all agents converge uniformly to a constant called asymptotic equilibrium mean (AEM) as the number of agents becomes large.  The AEM, which characterizes the asymptotic behavior of each agent's NE strategy, depends on the limit of (deterministic) empirical distribution of $\left\{\alpha_i\right\}_i$. 

 Next, three approximation methods are proposed to reduce the computational/communication complexity of finding the AEM at a processing center. The AEM is the solution of an implicit equation which involves integration with respect to the limiting distribution of $\left\{\alpha_i\right\}_i$ (see equation \eqref{EQ: MF} for more details). The first approximation method reduces the computational complexity of solving the implicit equation by quantizing the underlying distribution function. Under this method, the processing center requires the knowledge of the limiting distribution of $\left\{\alpha_i\right\}_i$.

The second method approximates the AEM using the empirical distribution of $\left\{\alpha_i\right\}_i$. This method is suitable for applications in which the limiting distribution of $\left\{\alpha_i\right\}_i$ is not known. However, it requires that each agent $i$ to transmit its $\alpha_i$ to the processing center which might entail substantial communication and storage costs when the agents' population is large. Under the third approximation method, the processing center receives the quantized versions of $\left\{\alpha_i\right\}_i$ and approximates the AEM using the quantized values. The accuracy of the proposed approximation methods is investigated.  The proposed approximation methods significantly reduce the (computation/communication) cost of evaluating the equilibrium of large scale heterogeneous games using modern computation technologies, \emph{e.g.,} cloud computing.   

The rest of this paper is structured as follows. Section \ref{Sec: SM} describes our system model and assumptions. Section \ref{Sec: EA} presents our results on the asymptotic equilibrium behavior of agents. Section \ref{Sec: PC} studies the AEM approximation problem under communication and computation constraints. Section \ref{Sec: NR} presents the numerical results and Section \ref{Sec: Conc} concludes the paper.
\section{System Model}\label{Sec: SM}


Consider a static non-cooperative game with $n$ agents. Let $\utility$ denote the utility function and $x_{i}$ the action (decision) of agent $i$. The utility function is parameterized by $\alpha_i\in\R$ and
\begin{eqnarray}
z_{i,n}=\frac{1}{n-1}\sum_{j\neq i}x_{j}\nonumber
\end{eqnarray}
is the \emph{mean variable} observed by agent $i$, \emph{i.e.,} the average of actions of all agents except the $i$-the one. The action space of each agent is assumed to be limited to the interval $\left[a,b\right]\subset\R$ as a starting point, which we will generalize to jointly convex and compact constraint sets in the future.
We refer to the defined $n$-player non-cooperative game among agents as $\mathcal{G}_n$.

Note that the utility functions of agents belong to the parametrized family of functions $\left\{u\paren{\cdot,\cdot,\alpha}\right\}_{\alpha\in\R}$. 
The interpretation of the $\alpha$ parameter depends on the underlying application. For example, consider the power control problem in downlink code-division-multiple-access networks in a slow fading environment with non-orthogonal spreading codes and  match filter receivers. In this application, the parameter $\alpha_i$ may represent the channel power gain between the base station and receiver (agent) $i$. 

In non-cooperative games, each agent selfishly maximizes its own utility function, given the actions (strategies). 
 Thus, agent $i$ is interested in the solution of the following optimization problem:
  \begin{eqnarray} \label{e:agentopt1}
\underset{a\leq x_{i}\leq b}{\rm maximize} & \utility . 
 \end{eqnarray}

In this paper, the Nash equilibrium (NE) is considered as the solution of the non-cooperative game among agents. Let $x^\star_{i,n}$ be the NE strategy of agent $i$ in the game $\mathcal{G}_n$ with $n$ players.
At the NE point of the game, no agent has motivation to unilaterally deviate his strategy given the NE strategy of the other agents, \emph{i.e.},
  \begin{eqnarray}
x^\star_{i,n}=\arg\underset{a\leq x_{i}\leq b}{\max} & u\paren{x_{i},z^\star_{i,n},\alpha_i}, \forall i.\nonumber
 \end{eqnarray}
where
\begin{eqnarray}
z^\star_{i,n}=\frac{1}{n-1}\sum_{j\neq i}x^\star_{j,n}\nonumber
\end{eqnarray}
Next, the best response function of each agent is introduced which plays an important role in our final results. The best response of agent $i$ to $z_{i,n}$, which follows from the solution of (\ref{e:agentopt1}), is defined as
  \begin{eqnarray}
{\rm Br}\paren{z_{i,n},\alpha_i}=\arg\underset{a\leq x_{i}\leq b}{\max} & u\paren{x_{i},z_{i,n},\alpha_i}, \forall i.\nonumber
 \end{eqnarray}
It is well known that, any  intersection of the best response functions of agents is a NE, \emph{i.e.,}
\begin{align}
x^\star_{i,n}={\rm Br}\paren{z^\star_{i,n},\alpha_i},\quad  1\leq i\leq n.
\end{align}

This model will be used to characterize the asymptotic behavior of the NE strategies of agents in the game $\mathcal{G}_n$ as the number of agents becomes large. In practice, the results obtained in the paper
enables a centralised \textit{processing center} compute the approximate NE of a large-scale game with
many agents. The processing center approximately computes the NE based on either a priori information available,
e.g. coming from a previous instance of a similar game, or explicit information communicated by the
agents to the processing center about their individual utilities, in the form of (quantized) $\alpha_i$s.
The communication and computational model is visualised in Figure~\ref{fig:model}. Note that, from
a technological perspective, this setup fits well into prevalent cloud computing models.
\begin{figure}[!htp]
\centering{\includegraphics[width=0.7\columnwidth]{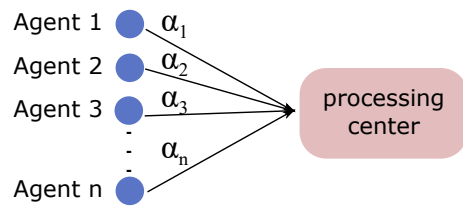}}
\caption{The communication and computational model.} \label{fig:model}
\end{figure}

The following additional assumptions are imposed on the utility functions of agents and the parameters $\left\{\alpha_i\right\}_i$.
\begin{enumerate}
\item The best response function  is Lipschitz in $z$ with the Lipschitz constant $L_z< 1$, \emph{i.e.,} there exists a positive constant $L_z<1$ such that
  \begin{align}
&\abs{{\rm Br}\paren{z,\alpha}-{\rm Br}\paren{z^\prime,\alpha}}\leq L_z {\abs{z-z^\prime}} \quad \forall z,z^\prime,\alpha\in\R  \nonumber
 \end{align}
\item The best response function is Lipschitz in the parameter $\alpha$ with the Lipschitz constant $L$, \emph{i.e.,} we can find a positive constant $L$ such that 
  \begin{align}
\abs{{\rm Br}\paren{z,\alpha}-{\rm Br}\paren{z,\alpha^\prime}}\leq L {\abs{\alpha-\alpha^\prime}} \quad \forall z,\alpha,\alpha^\prime\in\R\nonumber
 \end{align}
\item The empirical distribution of the \emph{deterministic} parameters $\left\{\alpha_i\right\}_i$ converges to a continuous distribution function $P\paren{\alpha}$ as the number of agents increases. Let $P_n\paren{\alpha}$ be the (deterministic) empirical distribution function of $\left\{\alpha_i\right\}_{i=1}^n$, \emph{i.e.,}
  \begin{eqnarray}
P_n\paren{\alpha}=\frac{1}{n}\sum_{i=1}^n\I{\alpha_i\leq \alpha}.\nonumber
\end{eqnarray}
Then, we have
 \begin{eqnarray}\label{Eq: Lim-Dis}
\lim_{n\rightarrow \infty}P_n\paren{\alpha}=P\paren{\alpha}\quad \forall x\in\R.
\end{eqnarray}
\end{enumerate}
Assumptions 1 and 2 imply that the best response of each agent $i$ cannot change arbitrarily fast as its mean variable $z_{i,n}$ or its parameter $\alpha_i$ vary.

\begin{remark}
We note that our results in this paper do not entirely depend on the Assumption 3. In the case that the limiting distribution $P\paren{\alpha}$ is not known or does not exist, we use the empirical distribution function of $P_n\paren{\alpha}$ to study the asymptotic behavior of the NE of the game $\mathcal{G}_n$ (see Proposition \ref{Lem: U-PMFA} for more details).
\end{remark}



\section{Equilibrium Analysis}\label{Sec: EA}
In this section, we first study the existence and uniqueness of the NE of the game $\mathcal{G}_n$. We next characterize the asymptotic behavior of NE strategies of agents. 

Next theorem establishes the existence and uniqueness of the NE of the game $\mathcal{G}_n$.
\begin{lemma}\label{Theo: Nash-uniq}
For $n\geq 2$, the game $\mathcal{G}_n$ admits a unique Nash equilibrium.
\end{lemma}
\begin{proof}
See Appendix \ref{App: Nash-uniq}.
\end{proof}
Lemma \ref{Theo: Nash-uniq} is proved by showing that the best response function of the game forms a contraction mapping.
We  next characterize the asymptotic behavior of the (equilibrium)  mean variables $z^\star_{i,n}$s as the number of agents becomes large. To this end, consider the sequence of games $\left\{\mathcal{G}_n\right\}_{n=2}^\infty$ with increasing number of agents such that the equation \eqref{Eq: Lim-Dis} holds. For this sequence of games, we define the asymptotic equilibrium mean (AEM), \emph{i.e.,} $z^\star_{\rm AEM}$, which captures the asymptotic behavior of equilibrium mean variables. 
\begin{definition}
Consider the sequence of games $\left\{\mathcal{G}_n\right\}_{n=2}^\infty$ which satisfies the equation \eqref{Eq: Lim-Dis}. Then,  the asymptotic equilibrium mean (AEM) $z^\star_{\rm AEM}$ is defined as the solution of following equation
\begin{align}\label{EQ: MF}
z^\star_{\rm AEM}&=\int_{-\infty}^\infty {\rm Br}\paren{z^\star_{\rm AEM},\alpha}dP\paren{\alpha}\nonumber\\
&=\ESI{\nu}{{\rm Br}\paren{z^\star_{\rm AEM},\nu}}
\end{align}
where $\nu$ is a generic random variable distributed according to $P\paren{\alpha}$.
\end{definition}

 The asymptotic behavior of the NE of the game $\mathcal{G}_n$ can also be characterized using $z^\star_{\rm AEM}$. Next, we define the notion of asymptotic Nash equilibrium (ANE) strategy of each agent.
\begin{definition}
Let $\left\{\mathcal{G}_n\right\}_{n=2}^\infty$ be a sequence of games such that the Assumption 3 in Section \ref{Sec: SM} holds. Let $z^\star_{\rm AEM}$ be the solution of \eqref{EQ: MF}. Then, the asymptotic Nash equilibrium (ANE) strategy of the agent $i$, denoted by $x^\star_{i,\rm ANE}$, is defined as
\begin{eqnarray}
x^\star_{i,\rm ANE}={\rm Br}\paren{z^\star_{\rm AEM},\alpha_i}
\end{eqnarray}
\end{definition}


The next lemma studies the uniqueness of $z^\star_{\rm AEM}$.
\begin{lemma}\label{Theo: MF-uniq}
The equation \eqref{EQ: MF} has a unique solution.
\end{lemma}
\begin{proof}
See Appendix \ref{App: MF-uniq}.
\end{proof}
The next theorem studies the asymptotic behavior of the triangular array of (equilibrium) mean variables $\left\{z^\star_{i,n}, 1\leq i\leq n\right\}_n$.
\begin{theorem}\label{Theo: MF}
Consider the sequence of games $\left\{\mathcal{G}_n\right\}_{n=2}^\infty$ such that the Assumption 3 in Section \ref{Sec: SM} holds. Let $z^\star_{i,n}$ denote the equilibrium mean variable of the agent $i$ at the NE of the game $\mathcal{G}_n$. Then, we have
\begin{eqnarray}
\lim_{n\rightarrow\infty} \sup_{1\leq i\leq n}\abs{z^\star_{i,n}-z^\star_{\rm AEM}}=0.\nonumber
\end{eqnarray}
\end{theorem}
\begin{proof}
Please see Appendix \ref{App: MF}.
\end{proof}
According to Theorem \ref{Theo: MF}, the mean variables of agents, at the NE of the game, converge uniformly to $z^\star_{\rm AEM}$ as the number of agents becomes large.

The next corollary shows that the NE strategy of each agent converges to its ANE strategy as the number of agents becomes large.
\begin{corollary}\label{Lem: N-MF}
Let $x^\star_{i,n}$ denote the NE strategy of agent $i$ in the game $\mathcal{G}_n$. Then, we have
\begin{eqnarray}
\lim_{n\rightarrow\infty} x^\star_{i,n}=x^\star_{i,\rm ANE}.\nonumber
\end{eqnarray}
\end{corollary}
\begin{proof}
The proof follows directly from Theorem \ref{Theo: MF} and the continuity of ${\rm Br}\paren{z,\alpha_i}$ in $z$.
\end{proof}

\section{Approximating AEM at the Processing Center}\label{Sec: PC}
Based on Theorem \ref{Theo: MF} and Corollary \ref{Lem: N-MF}, the AEM, \emph{i.e.,} $z^\star_{\rm AEM}$, can be used to approximate the NE strategies of agents in the game $\mathcal{G}_n$ as the number of agents  becomes large. However, to compute the AEM, the processing center requires $(i)$ the knowledge of the distribution function $P\paren{\alpha}$, $(ii)$ to solve the implicit equation \eqref{EQ: MF}. Even if the distribution function $P\paren{\alpha}$ is known at the processing center, the expectation on the right hand side of \eqref{EQ: MF} may not have a closed form expression, \emph{e.g.,} when the $P\paren{\alpha}$ and/or ${\rm Br}\paren{z,\alpha}$ have sophisticated forms.
In this section, we first propose a method for approximating the AEM when $P\paren{\alpha}$ is known at the processing center. We then consider the case that the processing center does not know $P_n\paren{\alpha}$ and propose two methods for approximating the AEM using the actual parameter values $\left\{\alpha_i\right\}_i$. Through this section, it is assumed that the processing center knows the general form of the utility functions.
\subsubsection{Approximating the AEM by Quantizing $P\paren{\alpha}$}

{Assume that the processing center has a priori information on $P\paren{\alpha}$, which may be due to domain knowledge or prior experiences, e.g. when solving a series of games over time with similar parameters. In such cases,
even if the solution is not exact, it can be used for example as a warm starting point as part of a multi-stage
solution process. However,}
 finding the exact solution of \eqref{EQ: MF} can be cumbersome when the best response function is non-linear in $\alpha$. Here, a quantization-based approximation method is proposed for computing an approximation of AEM. The proposed method replaces the integration of the best response function (with respect to the continuous distribution $P\paren{\alpha}$) by a summation {to ensure computational feasibility}.

 Let $Q\paren{\cdot}$ denote a quantization scheme for $P\paren{\alpha}$. Further, let ${\rm B}_m$  and $q_m$ denote the $m$th quantization cell and its corresponding cell center under the quantization scheme $Q\paren{\cdot}$. Next, we define an approximation of the AEM which is computed using the quantization scheme $Q\paren{\cdot}$.
\begin{definition}
Let $Q\paren{\cdot}$ be a quantization scheme for $P\paren{\alpha}$. Then, $Q\paren{\cdot}$-based approximation of ${z}^\star_{\rm AEM}$, denoted by $\hat{z}^\star_{\rm AEM}$,  is defined as the solution of the following equation
 \begin{align}\label{EQ: QAMF}
\hat{z}^\star_{\rm AEM}&=\sum_m{\rm Br}\paren{\hat{z}^\star_{\rm AEM},q_m}\PRP{{\rm B}_m}\nonumber\\
&=\ESI{\nu}{{\rm Br}\paren{\hat{z}^\star_{\rm AEM},Q\paren{\nu}}}
\end{align}
where $\nu$ is a random variable with the distribution $P\paren{\alpha}$.
\end{definition}
 Let the map $\hat{F}\paren{z}$ be defined as $\hat{F}\paren{z}=\ESI{\nu}{{\rm Br}\paren{z,Q\paren{\nu}}}$. Similar to the proof of Lemma \ref{Theo: MF-uniq}, it can be shown that $\hat{F}\paren{z}$ is a contraction and admits a unique fixed point. Thus, $\hat{z}^\star_{\rm AEM}$ can be computed using fixed point iterations. Next Proposition derives an upper bound on the distance between the $\hat{z}^\star_{\rm AEM}$ and ${z}^\star_{\rm AEM}$
\begin{proposition}\label{Lem: QAMF-E}
Let $Q\paren{\cdot}$ be a quantization scheme for $P\paren{\alpha}$ and $\hat{z}^\star_{\rm AEM}$ be the $Q\paren{\cdot}$-based approximation of $z^\star_{\rm AEM}$, \emph{i.e.,} the solution of \eqref{EQ: QAMF}. Then, we have
\begin{eqnarray}
\abs{{z}^\star_{\rm AEM}-\hat{z}^\star_{\rm AEM}}\leq \frac{L}{1-L_z}\ESI{\nu}{\abs{\nu-Q\paren{\nu}}}.
\end{eqnarray}
\end{proposition}
\begin{proof}
See Appendix \ref{App: QAMF-E}.
\end{proof}
Proposition \ref{Lem: QAMF-E} establishes an upper bound on the distance between $\hat{z}^\star_{\rm AEM}$ and the $z^\star_{\rm AEM}$ which depends on the Lipschitz constants $L,L_z$, and the quantization error. Thus, according to this proposition, the approximation error can be made arbitrarily small by increasing the number of quantization levels, { which trades-off computational requirements (memory, processing power) with accuracy.}

\subsubsection{Approximating the AEM Using $P_n\paren{\alpha}$}
The $Q\paren{\cdot}$-based approximation of $z^\star_{\rm AEM}$ requires that the knowledge of the distribution function $P\paren{\alpha}$ to be available at the processing center before solving the game. However, in many applications, the distribution $P\paren{\alpha}$ may not be known a priori. Here, an alternative method is proposed which approximates $z^\star_{\rm AEM}$ using the empirical distribution of $\left\{\alpha_i\right\}_i$, {based on the information received by the processing center from individual agents}. To this end, it is assumed that each agent $i$ transmits its parameter $\alpha_i$ to the the processing center. {Note that, as a starting point, it is assumed here that the agents send real-valued $\alpha$ values.} Upon receiving $\left\{\alpha_i\right\}_i$, the processing center will compute $P_n\paren{\alpha}$.

\begin{definition}
Consider the game $\mathcal{G}_n$ in which the empirical distribution of $\left\{\alpha_i\right\}_i$ is given by $P_n\paren{\alpha}$. Then, the $P_n\paren{\alpha}$-based approximation of the $z^\star_{\rm AEM}$, denoted by $\hat{z}^\star_{n,\rm AEM}$, is defined as the solution of the following equation
 \begin{align}\label{EQ: PAMF}
\hat{z}^\star_{n,\rm AEM}&=\frac{1}{n}\sum_{i=1}^n{\rm Br}\paren{\hat{z}^\star_{n,\rm AEM},\alpha_i}\nonumber\\
&=\ESI{\hat{\nu}_n}{{\rm Br}\paren{\hat{z}^\star_{n,\rm AEM},\hat{\nu}_n}}
\end{align}
where $\hat{\nu}_n$ is an auxiliary random variable distributed over $\left\{\alpha_i\right\}_{i=1}^n$ according to $P_n\paren{\alpha}$ and independent of $\nu$.
\end{definition}
It is straight forward to verify that the solution of \eqref{EQ: PAMF} is unique and can be found using a fixed point iteration. Note that $\hat{z}^\star_{n,\rm AEM}$ depends on $n$ and changes as the number of agents varies. Next Proposition derives an upper bound on the distance between $z^\star_{\rm AEM}$ and $\hat{z}^\star_{n,\rm AEM}$.
\begin{proposition}\label{Lem: PAMF}
Let $\hat{z}^\star_{n,\rm AEM}$ denote the $P_n\paren{\alpha}$-based approximation of  $z^\star_{\rm AEM}$.  Then, we have
\begin{eqnarray}
\abs{\hat{z}^\star_{n,\rm AEM}-z^\star_{\rm AEM}}\leq \frac{L}{1-L_z}\ESI{\nu,\hat{\nu}_n}{\abs{\nu-\hat{\nu}_n}}.
\end{eqnarray}
where $\nu$ and $\hat{\nu}_n$ are two independent random variables distributed according to $P\paren{\alpha}$ and $P_n\paren{\alpha}$, respectively.
\end{proposition}
\begin{proof}
Please see Appendix \ref{App: PAMF}.
\end{proof}
According to Proposition \ref{Lem: PAMF}, the error of the $P_n\paren{\alpha}$-based method is upper bounded by the $L_1$ distance between the random variables $\nu$ and $\hat{\nu}_n$. Note that $\hat{\nu}_n$ converges in distribution to $\nu$ since $P_n\paren{\alpha}$ converges to $P\paren{\alpha}$ as $n$ becomes large. Thus, under mild regularity assumptions, \emph{e.g.,} uniform integrability of $\left\{\hat{\nu}_n\right\}_n$, $\hat{\nu}_n$ converges in mean to $\nu$ \cite{PB95}. This implies that the distance between  $z^\star_{\rm AEM}$ and $\hat{z}^\star_{n,\rm AEM}$ will be arbitrarily small as the number of agents becomes large if the sequence of distribution functions  $\left\{P_n\paren{\alpha}\right\}_n$ is well-behaved.

The next Proposition studies the closeness of $\hat{z}^\star_{n,\rm AEM}$ and the triangular array  of (equilibrium) mean variables $\left\{z^\star_{i,n}, 1\leq i\leq n\right\}_n$.
\begin{proposition}\label{Lem: U-PMFA}
Consider the non-cooperative game $\mathcal{G}_n$. Then, the distance between the (equilibrium) mean variables $\left\{z^\star_{i,n}\right\}_i$ and the solution of the $P_n\paren{\alpha}$-based approximation can be uniformly bounded as
\begin{eqnarray}
\max_{1\leq i\leq n}\abs{\hat{z}^\star_{n,\rm AEM}-{z}^\star_{i,n}}=\BO{\frac{1}{n}}.
\end{eqnarray}
\end{proposition}
\begin{proof}
See Appendix \ref{App: U-PMFA}.
\end{proof}
Proposition \ref{Lem: U-PMFA} implies that, for $n$ large enough, the mean variables $\left\{z^\star_{i,n}\right\}_i$ of the game $\mathcal{G}_n$ can be well approximated with $\hat{z}^\star_{n,\rm AEM}$. Note that Proposition \ref{Lem: U-PMFA} is proved without assuming that the sequence of empirical distribution functions $\left\{P_n\paren{\alpha}\right\}_n$ is convergent. Thus, according to this proposition, $\hat{z}^\star_{n,\rm AEM}$ can be used to obtain an accurate approximation of mean variables even if  the limiting distribution $P\paren{\alpha}$ in unknown or  it does not exist, \emph{i.e.,} when the sequence $\left\{P_n\paren{\alpha}\right\}_n$ is not convergent.

\subsubsection{Approximating AEM Using Quantized Prameters}
In order to compute the $P_n\paren{\alpha}$-based approximation, the processing center requires the knowledge of $\left\{\alpha_i\right\}_i$. Since the parameters $\left\{\alpha_i\right\}_i$ are real valued, transmitting each $\alpha_i$ with high accuracy will impose a large communication cost on each agent as well as a large storage cost on the processing  center.  Further, the total communication  (storage) cost of the $P_n\paren{\alpha}$-based method increases  linearly with the number of agents. Hence, the $P_n\paren{\alpha}$-based approximation will become impractical  in large population games especially when $P_n\paren{\alpha}$ changes regularly with time, \emph{e.g.,}  when agents regularly enter or leave the game.

 To {address} this problem, an approximation method is proposed in which the processing center computes an approximate AEM using quantized values of $\left\{\alpha_i\right\}_i$. Under the proposed method, each agent $i$ transmits the quantized version of its $\alpha_i$ to the processing center. Then, the processing center computes an approximation of $z^\star_{\rm AEM}$ using the empirical distribution of the quantized $\left\{\alpha_i\right\}_i$.  We refer to this approximation method as $\alpha_i$-quantized approximation. To this end, consider a quantization scheme on $\R$ with $k$ quantization cells denoted by $\left\{C_i\right\}_{i=1}^k$, \emph{i.e.,} $\left\{C_i\right\}_{i=1}^k$ form a non-overlapping partitioning of the real line. Also, let ${c}_i\in C_i$ denote the representative of the cell ${C}_i$.  Under the $\alpha_i$-quantized approximation method, the agent $i$ transmits $c_i$ to the processing center if $\alpha_i$ belongs to $C_i$. Thus, each agent only needs to communicate $\log_2k$ bits to the processing center in order to convey its quantized parameter. We assume that each $C_i$ is right-closed, thus $\sup C_i$ belongs to $C_i$.  

To study the performance of the $\alpha_i$-quantized approximation method, it is helpful to introduce two auxiliary random variables $\hat{\nu}_n$ and $\tilde{\nu}_n$. Let $\hat{\nu}_n$  be a random variable distributed over $\left\{\alpha_i\right\}_{i=1}^n$ according to a uniform distribution. The random variable $\tilde{\nu}_n$, which is distributed over $\left\{{c}_i\right\}_{i=1}^k$, is constructed using $\hat{\nu}_n$ as follows
\begin{align}
\tilde{\nu}_n={c}_i \quad \textit{if $\hat{\nu}_n\in{C}_i$}
\end{align}
Note that the distribution function of $\tilde{\nu}_n$, denoted as $\tilde{P}_n\paren{x}$,  can be written as
\begin{align}
\tilde{P}_n\paren{x}=\sum_{j=1}^k\I{{c}_j\leq x<{c}_{j+1}}\frac{1}{n}\sum_{i=1}^n\I{\alpha_i\leq \sup C_j}\nonumber
\end{align}
where ${c}_{k+1}=\infty$.

Next the notion of $\alpha_i$-quantized approximation of $z^\star_{\rm AEM}$ is defined.
\begin{definition}
The $\alpha_i$-quantized  approximation of $z^\star_{\rm AEM}$, denoted by $\tilde{z}^\star_{n,\rm AEM}$, is defined as the solution of the following equation
 \begin{align}\label{EQ: CMFA}
\tilde{z}^\star_{n,\rm AEM}&=\sum_{j=1}^k{\rm Br}\paren{\tilde{z}^\star_{n,\rm AEM},{c}_j}\frac{1}{n}\sum_{i=1}^n\I{\alpha_i\in{C}_j}\nonumber\\
&=\ESI{\tilde{\nu}_n}{{\rm Br}\paren{\tilde{z}^\star_{n,\rm AEM},\tilde{\nu}_n}}\nonumber
\end{align}
where $\sum_i\I{\alpha_i\in{C}_j}$ is the number of agents which transmit $c_j$ to the processing center.
\end{definition}
To compute $\tilde{z}^\star_{n,\rm AEM}$, the processing center requires to store $\left\{c_j\right\}_{j=1}^k$ and  $\left\{\sum_{i=1}^n\I{\alpha_i\in C_j}\right\}_{j=1}^k$ which can be stored with at most $k\log_2k+k\log_2 n$ bits. However, to compute $\hat{z}^\star_{n,\rm AEM}$, the processing center needs to store $n$ real variables. 

 Next Proposition studies the closeness of the solutions of the $\alpha_i$-quantized and $P_n\paren{\alpha}$-based methods.
\begin{proposition}\label{Lem: CMFA}
Let $\tilde{z}^\star_{n,\rm AEM}$ denote the solution of the $\alpha_i$-quantized approximation method. Then, we have
\begin{eqnarray}
\abs{\tilde{z}^\star_{n,\rm AEM}-\hat{z}^\star_{n,\rm AEM}}\leq \frac{L}{1-L_z}\ESI{\hat{\nu}_n}{\abs{\tilde{\nu}_n-\hat{\nu}_n}}.
\end{eqnarray}
where the expectation is with respect to the distribution of $\hat{\nu}_n$ as $\tilde{\nu}_n$ is a function of $\hat{\nu}_n$.
\end{proposition}
\begin{proof}
See Appendix \ref{App: CMFA}.
\end{proof}
According to the Proposition \ref{Lem: CMFA}, the distance between the solutions of the $P_n\paren{\alpha}$-based and $\alpha_i$-quantized methods is controlled by $L$, $L_z$ and the $L_1$ distance between the random variables $\hat{\nu}_n$ and $\tilde{\nu}_n$. Similar to $P_n\paren{\alpha}$-based AEM approximation, the  $\alpha_i$-quantized AEM approximation method does not impose any assumption on the convergence of $\left\{P_n\paren{\alpha}\right\}_n$. Thus, this approximation method can be used to obtain an approximation of the NE when the distribution function $P\paren{\alpha}$ is unknown or does not exists. We note that one can combine the result of Proposition \ref{Lem: CMFA} with that in Propositions \ref{Lem: PAMF}, to obtain a bound on the difference between $\tilde{z}^\star_{n,\rm AEM}$ and AEM.


\section{Numerical Results}\label{Sec: NR}
In this section, we consider a non-cooperative game and numerically analyze the proximity of the equilibrium behavior of agents with the asymptotic equilibrium mean (AEM) and the solution of the $P_n\paren{\alpha}$-based approximation method.  To this end, consider a non-cooperative game in which the utility function of agent $i$ is given by
\begin{align}
\utility=-\paren{x_i-\sqrt{z_{i,n}^2+\alpha_i^2}}^2-2x_i\nonumber.
\end{align}
It is assumed that the empirical distribution of the parameters $\left\{\alpha_i\right\}_i$ converges to the distribution function of a Gaussian random variable with zero mean and of variance 4. The action space of each agent is limited to the interval $\left[0.5,20\right]$.

Fig. \ref{F1} shows the Nash equilibrium (NE) strategy of Agent 1  and its corresponding equilibrium mean variable as a function of the number of agents. According to Fig.\ref{F1}, the asymptotic Nash equilibrium strategy of Agent 1   and the AEM characterize the equilibrium behavior of Agent 1  as $n$ becomes large. More precisely, based on this figure, the NE strategy of Agent 1  converges to its asymptotic Nash equilibrium strategy, \emph{i.e.,} $x^\star_{1,\rm ANE}$, as the number of agents becomes large. Similarly, the equilibrium mean variable of Agent 1  converges to the AEM as depicted in Fig. \ref{F1}.
\begin{figure}[!t]
\centering{\includegraphics[scale=0.44]{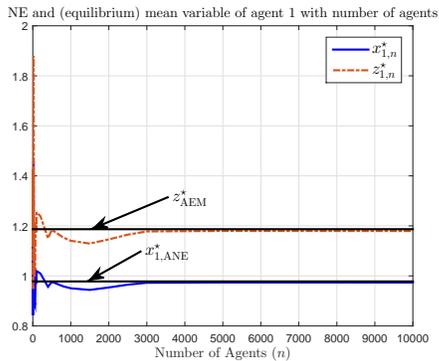}}
\caption{The behavior of the Nash equilibrium strategy and equilibrium mean variable of Agent 1  with the number of agents.} \label{F1}
\end{figure}
\begin{figure}[!t]
\centering{\includegraphics[scale=0.6]{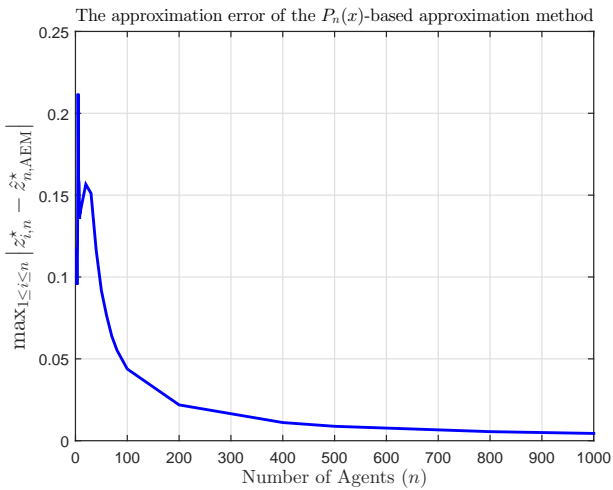}}
\caption{The error of the $P_n\paren{\alpha}$-based approximation method with the number of agents.} \label{F2}
\end{figure}

Fig. \ref{F2} illustrates the error of the $P_n\paren{\alpha}$-based approximation method, \emph{i.e.,} $\max_{1\leq i\leq n}\abs{z^\star_{i,n}-\hat{z}^\star_{n,\rm AEM}}$, with the number of agents. According to this figure, the error of the $P_n\paren{\alpha}$-based approximation method decays to zero as the number of agents becomes large. Moreover, based on this figure, the $P_n\paren{\alpha}$-based approximation method can be used to obtain an accurate (uniform) approximation of the equilibrium mean variables of agents even for moderate number of agents.

%

\section{CONCLUSIONS}\label{Sec: Conc}
In this paper, we studied the equilibrium behavior of agents in a class of large-scale heterogeneous mean-field games in which the utility function of each agent depends on its action, average of other agents' actions and a fixed parameter. First, it was shown that the mean variables of all agents converge to a constant, named asymptotic equilibrium mean (AEM), as the number of agents becomes large. Then, the problem of approximating the AEM at a processing center was considered and three approximation methods were proposed which allow the processing center to approximate the AEM under limited computational and communication resources.

\appendices
\section{Proof of Lemma \ref{Theo: Nash-uniq}}\label{App: Nash-uniq}
Lemma \ref{Theo: Nash-uniq} is proved by showing that the best response function of the game $\mathcal{G}_n$ forms a contraction map in the infinity norm. To this end, let $X=\left[x_1,\cdots,x_n\right]^\top$ and $X^\prime=\left[x^\prime_1,\cdots,x^\prime_n\right]^\top$ be two $n$-dimensional vectors. Also, let $Z=\left[z_{i,n}\right]_i$, and $Z^\prime=\left[z^\prime_{i,n}\right]_i$ where $z_{i,n}=\frac{1}{n-1}\sum_{j\neq i}x_j$ and $z^\prime_{i,n}=\frac{1}{n-1}\sum_{j\neq i}x^\prime_j$.
Let $G\paren{X}: \left[a,b\right]^n\mapsto\left[a,b\right]^n$ be the best response of agents to $X\in\left[a,b\right]^n$, \emph{i.e.,} $G\paren{X}=\left[{\rm Br}\paren{z_{i,n},\alpha_i}\right]_i$. Note that ${Z-Z^\prime}$ can be written as
\begin{align}
{Z-Z^\prime}=\frac{1}{n-1}\paren{\1_n-I_n}\paren{X-X^\prime}
\end{align}
where $\1_n$ is an $n$-by-$n$ matrix with all entries equal to one and $I_n$ is an $n$-by-$n$ identity matrix. Thus, we have $\norm{Z-Z^\prime}{\infty}\leq \norm{X-X^\prime}{\infty}$. The infinity norm of  $G\paren{X}-G\paren{X^\prime}$ can be upper bounded as
 \begin{align}
\norm{G\paren{X}-G\paren{X^\prime}}{\infty}&=\max_i\abs{{\rm Br}\paren{z_{i,n},\alpha_i}-{\rm Br}\paren{z^\prime_{i,n},\alpha_i}}\nonumber\\
&\stackrel{(a)}{\leq} L_z \max_i\abs{z_{i,n}-z^\prime_{i,n}}\nonumber\\
&=L_z\norm{Z-Z^\prime}{\infty}\nonumber\\
&\leq L_z\norm{X-X^\prime}{\infty}\nonumber
\end{align}
where $(a)$ follows from the fact that the best response function is Lipschitz in $z$. This implies that the $G\paren{X}$ is a contraction as $L_z< 1$, hence, it admits a unique fixed point \cite{BT97}. Since Nash equilibria are fixed points of the best response function, the game $\mathcal{G}_n$ admits a unique NE.
\section{Proof of Lemma \ref{Theo: MF-uniq}}\label{App: MF-uniq}
To prove this lemma, let $F\paren{z}=\ESI{\nu}{{\rm Br}\paren{z,\nu}}$. We show that $F\paren{z}:\left[a,b\right]\mapsto\left[a,b\right]$ is a contraction. 
Note that $\abs{F\paren{z}-F\paren{z^\prime}}$ can be upper bounded as
\begin{align}
\abs{F\paren{z}-F\paren{z^\prime}}&=\abs{\ESI{\nu}{{\rm Br}\paren{z,\nu}}-\ESI{\nu}{{\rm Br}\paren{z^\prime,\nu}}}\nonumber\\
&\leq \ESI{\nu}{\abs{{\rm Br}\paren{z,\nu}-{\rm Br}\paren{z^\prime,\nu}}}\nonumber\\
&\leq \ESI{\nu}{L_z\abs{z-z^\prime}}\nonumber\\
&=L_z\abs{z-z^\prime}\nonumber
\end{align}
Thus, the map $F\paren{z}$ has a unique fixed point.

\pagebreak
\section{Proof of Theorem \ref{Theo: MF}} \label{App: MF}

First note that $\abs{z^\star_{i,n}-z^\star_{\rm AEM}}$ can be upper bounded as
\begin{align}\label{Eq: AMF0000}
\abs{z^\star_{i,n}-z^\star_{\rm AEM}}&\stackrel{(a)}{\leq} \abs{z^\star_{i,n}-z^\star_{1,n}}+\abs{z^\star_{1,n}-z^\star_{\rm AEM}}\nonumber\\
&\stackrel{(b)}{=}\frac{1}{n-1}\abs{x^\star_{i,n}-x^\star_{1,n}}+\abs{z^\star_{1,n}-z^\star_{\rm AEM}}\nonumber\\
&\stackrel{(c)}{\leq} \frac{2\max\paren{\abs{a},\abs{b}}}{n-1}+\abs{z^\star_{1,n}-z^\star_{\rm AEM}}
\end{align}
where $(a)$ follows from the triangle inequality, $(b)$ follows from the fact that $\abs{z^\star_{i,n}-z^\star_{1,n}}=\frac{1}{n-1}\abs{x^\star_{i,n}-x^\star_{1,n}}$ and $(c)$ follows from the fact that $a\leq x^\star_{i,n}\leq b$  for all $i$.
 Next, we derive an expression for $z^\star_{\rm AEM}$. For $\delta>0$, let the set ${B}^\delta_j$ be defined as ${ B^\delta_j}=\left\{\paren{j-1}\delta< \nu\leq j\delta\right\}$ and $\bar{\nu}^\delta_j$ be a point in ${B}^\delta_j$. Next Lemma establishes an asymptotic expansion for $z^\star_{\rm AEM}$ in terms of $\delta$, ${B}^\delta_j$ and $\bar{\nu}^\delta_j$.
\begin{lemma}\label{Lem: AE-1}
For all $\delta>0$, $z^\star_{\rm AEM}$ can be written as
\begin{align}
z^\star_{\rm AEM}&=\sum_{j=-\infty}^\infty{\rm Br}\paren{z^\star_{\rm AEM},\bar{\nu}^\delta_j}\PRPI{\nu}{B^\delta_j}+\BO{L\delta}\nonumber
\end{align}
where $\PRPI{\nu}{B^\delta_j}=\PR{\paren{j-1}\delta< \nu\leq j\delta}$.
\end{lemma}
\begin{proof}
See Appendix \ref{App: Aux}.
\end{proof}
Next, we derive a similar asymptotic expansion for $z^\star_{1,n}$. To this end, let $\theta_n$ be a random variable distributed according to the empirical distribution of $\left\{\alpha_i\right\}_{i=2}^n$.
\begin{lemma}\label{Lem: AE-2}
For all $\delta>0$, $z^\star_{1,n}$ can be written as
\begin{align}
z^\star_{1,n}&=\sum_{j=-\infty}^\infty{\rm Br}\paren{z^\star_{1,n},\bar{\nu}^\delta_j}\PRPI{\theta_n}{B^\delta_j}+\BO{L\delta}\nonumber\\
&\hspace{4cm}+\BO{\frac{2L_z\max\paren{\abs{a},\abs{b}}}{n-1}}\nonumber
\end{align}
where $\PRPI{\theta_n}{B^\delta_j}=\PR{\paren{j-1}\delta< \theta_n\leq j\delta}$.
\end{lemma}
\begin{proof}
See Appendix \ref{App: Aux}.
\end{proof}
Combining Lemmas \ref{Lem: AE-1} and \ref{Lem: AE-2}, we have
\begin{align}\label{Eq: AMF-0}
&z^\star_{\rm AEM}-z^\star_{1,n}\nonumber\\
&=\sum_{j=-\infty}^\infty{\rm Br}\paren{z^\star_{\rm AEM},\bar{\nu}^\delta_j}\PRPI{\nu}{B^\delta_j}-{\rm Br}\paren{z^\star_{1,n},\bar{\nu}^\delta_j}\PRPI{\theta_n}{B^\delta_j}\nonumber\\
&\hspace{1.5cm}+\BO{L\delta}+\BO{\frac{2L_z\max\paren{\abs{a},\abs{b}}}{n-1}}\nonumber\\
&=\sum_{j=-\infty}^\infty{\rm Br}\paren{z^\star_{\rm AEM},\bar{\nu}^\delta_j}\paren{\PRPI{\nu}{B^\delta_j}-\PRPI{\theta_n}{B^\delta_j}}\nonumber\\
&\hspace{1.5cm}+\sum_{j=\infty}^\infty\paren{{\rm Br}\paren{z^\star_{\rm AEM},\bar{\nu}^\delta_j}-{\rm Br}\paren{z^\star_{1,n},\bar{\nu}^\delta_j}}\PRPI{\theta_n}{B^\delta_j}\nonumber\\
&\hspace{1.5cm}+\BO{L\delta}+\BO{\frac{2L_z\max\paren{\abs{a},\abs{b}}}{n-1}}
\end{align}
The absolute value of the second term on the right hand side of (the last equality in) \eqref{Eq: AMF-0}  can be upper bounded as
\begin{align}\label{Eq: AMF-00}
&\abs{\sum_{j=-\infty}^\infty\paren{{\rm Br}\paren{z^\star_{\rm AEM},\bar{\nu}^\delta_j}-{\rm Br}\paren{z^\star_{1,n},\bar{\nu}^\delta_j}}\PRPI{\theta_n}{B^\delta_j}}\nonumber\\
&\leq\sum_{j=-\infty}^\infty\abs{\paren{{\rm Br}\paren{z^\star_{\rm AEM},\bar{\nu}^\delta_j}-{\rm Br}\paren{z^\star_{1,n},\bar{\nu}^\delta_j}}}\PRPI{\theta_n}{B^\delta_j}\nonumber\\
&\stackrel{(a)}{\leq}\sum_{j=-\infty}^\infty L_z\abs{z^\star_{\rm AEM}-z^\star_{1,n}}\PRPI{\theta_n}{B^\delta_j}\nonumber\\
&=L_z\abs{z^\star_{\rm AEM}-z^\star_{1,n}}
\end{align}
where $(a)$ follows from the fact that ${\rm Br}\paren{z,\alpha}$ is Lipschitz in $z$. Combining equations \eqref{Eq: AMF-0} and \eqref{Eq: AMF-00}, we have
\begin{align}\label{Eq: AMF-000}
&\paren{1-L_z}\abs{z^\star_{\rm AEM}-z^\star_{1,n}}\nonumber\\
&\hspace{1.5cm}\leq\abs{\sum_{j=-\infty}^\infty{\rm Br}\paren{z^\star_{\rm AEM},\bar{\nu}^\delta_j}\paren{\PRPI{\nu}{B^\delta_j}-\PRPI{\theta_n}{B^\delta_j}}}\nonumber\\
&\hspace{2cm}+\BO{L\delta}+\BO{\frac{2L_z\max\paren{\abs{a},\abs{b}}}{n-1}}
\end{align}
Next lemma studies the asymptotic behavior of the first term in the right hand side of \eqref{Eq: AMF-000}.
\begin{lemma}\label{Lem: AE-3}
We have
\begin{eqnarray}
\lim_{n\rightarrow\infty}\abs{\sum_{j=-\infty}^\infty{\rm Br}\paren{z^\star_{\rm AEM},\bar{\nu}^\delta_j}\paren{\PRPI{\nu}{B^\delta_j}-\PRPI{\theta_n}{B^\delta_j}}}=0\nonumber
\end{eqnarray}
\end{lemma}
\begin{proof}
See Appendix \ref{App: Aux}.
\end{proof}
Combining equations \eqref{Eq: AMF0000}, \eqref{Eq: AMF-000} and Lemma \ref{Lem: AE-3}, we have
\begin{align}
\lim_{n\rightarrow\infty} \sup_{1\leq i\leq n}\abs{z^\star_{i,n}-z^\star_{\rm AEM}}\leq \BO{L\delta}
\end{align}
The desired result follows from the fact that the positive constant $\delta$ can be arbitrarily close to zero.
\section{Proof of Auxiliary Lemmas} \label{App: Aux}
\subsection{Proof of Lemma \ref{Lem: AE-1}}\label{App: Lem: AE-1}
Note that $z^\star_{\rm AEM}$ can be written as
\begin{align}\label{Eq: AMF-1}
z^\star_{\rm AEM}&=\ESI{\nu}{{\rm Br}\paren{z^\star_{\rm AEM},\nu}}\nonumber\\
&=\ESI{\nu}{\sum_{j=-\infty}^\infty{\rm Br}\paren{z^\star_{\rm AEM},\nu}\I{B^\delta_j}}\nonumber\\
&\stackrel{(a)}{=}\sum_{j=-\infty}^\infty\ESI{\nu}{{\rm Br}\paren{z^\star_{\rm AEM},\nu}\I{B^\delta_j}}\nonumber\\
&=\sum_{j=-\infty}^\infty\ESI{\nu}{\paren{{\rm Br}\paren{z^\star_{\rm AEM},\nu}-{\rm Br}\paren{z^\star_{\rm AEM},\bar{\nu}^\delta_j}\right.\right.\nonumber\\
&\hspace{2cm}\left.\left.+{\rm Br}\paren{z^\star_{\rm AEM},\bar{\nu}^\delta_j}}\I{B^\delta_j}}\nonumber\\
&=\sum_{j=-\infty}^\infty\ESI{\nu}{\paren{{\rm Br}\paren{z^\star_{\rm AEM},\bar{\nu}^\delta_j}}\I{B^\delta_j}}\nonumber\\
&+\sum_{j=-\infty}^\infty\ESI{\nu}{\paren{{\rm Br}\paren{z^\star_{\rm AEM},\nu}-{\rm Br}\paren{z^\star_{\rm AEM},\bar{\nu}^\delta_j}}\I{B^\delta_j}}
\end{align}
where $(a)$ follows from Lebesgue dominated convergence Theorem. The second term in the right hand side of \eqref{Eq: AMF-1} can be upper bounded as
\begin{align}\label{Eq: AMF-2}
&\abs{\sum_{j=-\infty}^\infty\ESI{\nu}{\paren{{\rm Br}\paren{z^\star_{\rm AEM},\nu}-{\rm Br}\paren{z^\star_{\rm AEM},\bar{\nu}^\delta_j}}\I{B^\delta_j}}}\nonumber\\
&\leq \sum_{j=-\infty}^\infty\ESI{\nu}{\abs{\paren{{\rm Br}\paren{z^\star_{\rm AEM},\nu}-{\rm Br}\paren{z^\star_{\rm AEM},\bar{\nu}^\delta_j}}}\I{B^\delta_j}}\nonumber\\
 &\stackrel{(a)}{\leq}\sum_{j=-\infty}^\infty\ESI{\nu}{L\abs{\nu-\bar{\nu}^\delta_j}\I{B^\delta_j}}\nonumber\\
 &\stackrel{(b)}{\leq}\sum_{j=-\infty}^\infty\ESI{\nu}{L\delta\I{B^\delta_j}}\nonumber\\
 &=L\delta\sum_{j=-\infty}^\infty\PRPI{\nu}{B^\delta_j}\nonumber\\
 &=L\delta
\end{align}
where $(a)$ follows from the fact that ${\rm Br}\paren{z,\alpha}$ is Lipschitz in $\alpha$ and $(b)$ follows from the fact that $\abs{\nu-\bar{\nu}^\delta_j}\leq \delta$ when $\nu$ belongs to $B^\delta_j$. Combining equations \eqref{Eq: AMF-1} and \eqref{Eq: AMF-2}, we have
\begin{align}
z^\star_{\rm AEM}&=\sum_{j=-\infty}^\infty{\rm Br}\paren{z^\star_{\rm AEM},\bar{\nu}^\delta_j}\PRPI{\nu}{B^\delta_j}+\BO{L\delta}\nonumber
\end{align}
\subsection{Proof of Lemma \ref{Lem: AE-2}}
We can write $z^\star_{1,n}$ as
\begin{align}\label{Eq: App-Aux-1}
z^\star_{1,n}&=\frac{1}{n-1}\sum_{i=2}x^\star_{i,n}\nonumber\\
&=\frac{1}{n-1}\sum_{i=2}^n{\rm Br}\paren{z^\star_{i,n},\alpha_i}\nonumber\\
&=\frac{1}{n-1}\sum_{i=2}^n{\rm Br}\paren{z^\star_{1,n},\alpha_i}\nonumber\\
&\hspace{0.5cm}+\frac{1}{n-1}\sum_{i=2}^n{\rm Br}\paren{z^\star_{i,n},\alpha_i}-{\rm Br}\paren{z^\star_{1,n},\alpha_i}
\end{align}
The absolute value of the second term on the right hand side of \eqref{Eq: App-Aux-1} can be upper bounded as
\begin{align}
&\hspace{-0.5cm}\frac{1}{n-1}\abs{\sum_{i=2}^n{\rm Br}\paren{z^\star_{i,n},\alpha_i}-{\rm Br}\paren{z^\star_{1,n},\alpha_i}}\nonumber\\
&\hspace{1cm}\leq \frac{1}{n-1}\sum_{i=2}^n\abs{{\rm Br}\paren{z^\star_{i,n},\alpha_i}-{\rm Br}\paren{z^\star_{1,n},\alpha_i}}\nonumber\\
&\hspace{1cm}\stackrel{(a)}{\leq} \frac{1}{n-1}\sum_{i=2}^nL_z\abs{z^\star_{i,n}-z^\star_{1,n}}\nonumber\\
&\hspace{1cm}\stackrel{(b)}\leq \frac{1}{n-1}\sum_{i=2}^n\frac{L_z}{n-1}\abs{x^\star_{i,n}-x^\star_{1,n}}\nonumber\\
&\hspace{1cm}\stackrel{(c)}\leq \frac{1}{n-1}\sum_{i=2}^n\frac{L_z}{n-1}2\max\paren{\abs{a},\abs{b}}\nonumber\\
&\hspace{1cm}= \frac{2L_z\max\paren{\abs{a},\abs{b}}}{n-1}
\end{align}
where $(a)$ follows from the fact that ${\rm Br}\paren{z,\alpha}$ is Lipschitz in $z$, $(b)$ follows from the fact that $\abs{z^\star_{i,n}-z^\star_{1,n}}=\frac{1}{n-1}\abs{x^\star_{i,n}-x^\star_{1,n}}$ and $(c)$ follows from the fact that $a\leq x^\star_{i,n}\leq b$  for all $i$.
Thus, we have
\begin{align}\label{Eq: App-Aux-2}
z^\star_{1,n}&=\frac{1}{n-1}\sum_{i=2}^n{\rm Br}\paren{z^\star_{1,n},\alpha_i}+\BO{\frac{2L_z\max\paren{\abs{a},\abs{b}}}{n-1}}\nonumber
\end{align}
Let $\theta_n$ be a random variable distributed according to the empirical distribution of $\left\{\alpha_i\right\}_{i=2}^n$, \emph{i.e.,} a uniform distribution over $\left\{\alpha_i\right\}_{i=2}^n$. Then, we have
\begin{align}
\frac{1}{n-1}\sum_{i=2}^n{\rm Br}\paren{z^\star_{1,n},\alpha_i}=\ESI{\theta_n}{{\rm Br}\paren{z^\star_{1,n},\theta_n}}\nonumber
\end{align}
Following similar steps as those in the Appendix \ref{App: Lem: AE-1}, we have
\begin{align}
\ESI{\theta_n}{{\rm Br}\paren{z^\star_{1,n},\theta_n}}&=\sum_{j=-\infty}^\infty{\rm Br}\paren{z^\star_{1,n},\bar{\nu}^\delta_j}\PRPI{\theta_n}{B^\delta_j}\nonumber\\
&+\BO{L\delta}+\BO{\frac{2L_z\max\paren{\abs{a},\abs{b}}}{n-1}}\nonumber
\end{align}
where $\PRPI{\theta_n}{B^\delta_j}=\PR{(j-1)\delta<\theta_n\leq j\delta}$.
\subsection{Proof of Lemma \ref{Lem: AE-3}}
 Let $\theta_n$ be a random variable distributed according to the empirical distribution of $\left\{\alpha_i\right\}_{i=2}^n$. We define two axillary random variables $\gamma$ and $\gamma_n$ , distributed over $\left\{\bar{\nu}^\delta_j\right\}_{j=-\infty}^\infty$, with the following probability mass functions:
\begin{align}
\PRP{\gamma=\bar{\nu}^\delta_j}&:= \PRPI{\nu}{B^\delta_j}=\PR{(j-1)\delta<\nu\leq j\delta}\nonumber\\
\PRP{\gamma_n=\bar{\nu}^\delta_j}&:= \PRPI{\theta_n}{B^\delta_j}=\PR{(j-1)\delta<\theta_n\leq j\delta}\nonumber
\end{align}
Thus, we have
\begin{align}
\sum_{j=-\infty}^\infty{\rm Br}\paren{z^\star_{\rm AEM},\bar{\nu}^\delta_j}\paren{\PRPI{\nu}{B^\delta_j}-\PRPI{\theta_n}{B^\delta_j}}\nonumber\\
=\ESI{\gamma}{{\rm Br}\paren{z^\star_{\rm AEM},\gamma}}-\ESI{\gamma_n}{{\rm Br}\paren{z^\star_{\rm AEM},\gamma_n}}
\end{align}
The distribution functions of $\gamma$ and $\gamma_n$ can be written as
\begin{align}
\PR{\gamma\leq x}&=\sum_{j=-\infty}^\infty\I{\bar{\nu}^\delta_{j-1}\leq x<\bar{\nu}^\delta_j}\PR{\nu \leq \paren{j-1}\delta}\nonumber\\
\PR{\gamma_n\leq x}&=\sum_{j=-\infty}^\infty\I{\bar{\nu}^\delta_{j-1}\leq x<\bar{\nu}^\delta_j}\PR{\theta_n \leq \paren{j-1}\delta}\nonumber
\end{align}
 Since $P_n\paren{\alpha}$ converges in distribution to $P\paren{\alpha}$, $\theta_n$ also converges in distribution to $P\paren{\alpha}$. Recall that $\nu$ is distributed according to $P\paren{\alpha}$. The distribution function $P\paren{\alpha}$ is continuous, thus, we have
\begin{align}
\lim_{n\ra\infty}\PR{\theta_n \leq \paren{j-1}\delta}=\PR{\nu \leq \paren{j-1}\delta}\nonumber
\end{align}
as $\theta_n$ converges in distribution to $P\paren{\alpha}$ which implies that $\gamma_n$ converges in distribution to $\gamma$. Since ${\rm Br}\paren{z,\alpha}$ is bounded and continuous in $\alpha$, and $\gamma_n$ converges in distribution to $\gamma$, we have
\begin{align}
\lim_{n\rightarrow\infty}\ESI{\gamma_n}{{\rm Br}\paren{z^\star_{\rm AEM},\gamma_n}}=\ESI{\gamma}{{\rm Br}\paren{z^\star_{\rm AEM},\gamma}}\nonumber
\end{align}
\cite{PB95} which completes the proof.
\section{Proof of Proposition \ref{Lem: QAMF-E}} \label{App: QAMF-E}
Note that $\abs{z^\star_{\rm AEM}-\hat{z}^\star_{\rm AEM}}$ can be upper bounded as
\begin{align}
&\abs{z^\star_{\rm AEM}-\hat{z}^\star_{\rm AEM}}\nonumber\\
&\hspace{1cm}=\abs{\ESI{\nu}{{\rm Br}\paren{z^\star_{\rm AEM},\nu}-{\rm Br}\paren{\hat{z}^\star_{\rm AEM},Q\paren{\nu}}}}\nonumber\\
&\hspace{1cm}\leq \ESI{\nu}{\abs{{\rm Br}\paren{z^\star_{\rm AEM},\nu}-{\rm Br}\paren{{z}^\star_{\rm AEM},Q\paren{\nu}}}}\nonumber\\
&\hspace{1cm}+\ESI{\nu}{\abs{{\rm Br}\paren{z^\star_{\rm AEM},Q\paren{\nu}}-{\rm Br}\paren{\hat{z}^\star_{\rm AEM},Q\paren{\nu}}}}\nonumber\\
&\hspace{1cm}\stackrel{(a)}{\leq} L\ESI{\nu}{\abs{\nu-Q\paren{\nu}}}+L_z\abs{z^\star_{\rm AEM}-\hat{z}^\star_{\rm AEM}}\nonumber
\end{align}
where $(a)$ follows from the fact that the best response function is Lipschitz in $z$ and $\alpha$. The desired result follows from the last inequality above.
\section{Proof of Proposition \ref{Lem: PAMF}} \label{App: PAMF}
To prove this result, we rewrite $\abs{\hat{z}^\star_{n,\rm AEM}-{z}^\star_{\rm AEM}}$ as
\begin{align}
&\abs{z^\star_{\rm AEM}-\hat{z}^\star_{n,\rm AEM}}\nonumber\\
&\hspace{1cm}=\abs{\ESI{\nu}{{\rm Br}\paren{z^\star_{\rm AEM},\nu}}-\ESI{\hat{\nu}_n}{{\rm Br}\paren{\hat{z}^\star_{n,\rm AEM},\hat{\nu}_n}}}\nonumber\\
&\hspace{1cm}=\abs{\ESI{\nu,\hat{\nu}_n}{{\rm Br}\paren{z^\star_{\rm AEM},\nu}-{\rm Br}\paren{\hat{z}^\star_{n,\rm AEM},\hat{\nu}_n}}}\nonumber
\end{align}
where $\nu$ and $\hat{\nu}_n$ are two random variables distributed according to $P\paren{\alpha}$ and $P_n\paren{\alpha}$, receptively. The desired result follows from similar steps as those in the proof of Proposition \ref{Lem: QAMF-E}.
\section{Proof of Proposition \ref{Lem: U-PMFA}} \label{App: U-PMFA}
The distance between $z^\star_{i,n}$ and $\hat{z}^\star_{n,\rm AEM}$ can be upper bounded as
\begin{align}
&\abs{z^\star_{i,n}-\hat{z}^\star_{n,\rm AEM}}\nonumber\\
&\hspace{1cm}=\abs{\frac{1}{n-1}\sum_{j\neq i}{\rm Br}\paren{z^\star_{j,n},\alpha_j}-\frac{1}{n}\sum_j{\rm Br}\paren{\hat{z}^\star_{n,\rm AEM},\alpha_j}}\nonumber\\
&\hspace{1cm}\leq \abs{\frac{1}{n-1}\sum_{j\neq i}{\rm Br}\paren{z^\star_{j,n},\alpha_j}-{\rm Br}\paren{z^\star_{i,n},\alpha_j}}\nonumber\\
&\hspace{1cm}+ \abs{\frac{1}{n-1}\sum_{j\neq i}{\rm Br}\paren{z^\star_{i,n},\alpha_j}-\frac{1}{n}\sum_{j}{\rm Br}\paren{z^\star_{i,n},\alpha_j}}\nonumber\\
&\hspace{1cm}+\abs{\frac{1}{n}\sum_{j}{\rm Br}\paren{z^\star_{i,n},\alpha_j}-{\rm Br}\paren{\hat{z}^\star_{n,\rm AEM},\alpha_j}}\nonumber\\
&\hspace{1cm}\leq \frac{1}{n-1}\sum_{j\neq i}L_z\abs{z^\star_{j,n}-z^\star_{i,n}}+\frac{1}{n}\sum_{j}L_z\abs{z^\star_{i,n}-\hat{z}^\star_{n,\rm AEM}}\nonumber\\
&\hspace{1cm}+ \abs{\frac{1}{n}{\rm Br}\paren{z^\star_{i,n},\alpha_i}-\frac{1}{n\paren{n-1}}\sum_{j\neq i}{\rm Br}\paren{z^\star_{i,n},\alpha_j}}
\end{align}
Note that $a\leq {\rm Br}\paren{z,\alpha}\leq b$. Also, $\abs{z^\star_{j,n}-z^\star_{i,n}}$ can be upper bounded as
\begin{align}
\abs{z^\star_{j,n}-z^\star_{i,n}}&=\frac{1}{n-1}\abs{x^\star_{j,n}-x^\star_{i,n}}\nonumber\\
&\leq \frac{2\max\paren{\abs{b},\abs{a}}}{n-1}\nonumber
\end{align}
as $a\leq x_i\leq b$. Thus, $\abs{z^\star_{i,n}-\hat{z}^\star_{n,\rm AEM}}$ can be further upper bounded as
\begin{align}
&\abs{z^\star_{i,n}-\hat{z}^\star_{n,\rm AEM}}\leq 2\max\paren{\abs{b},\abs{a}}\paren{\frac{L_z}{n-1}+\frac{1}{n}}\nonumber\\
&\hspace{4cm}+L_z\abs{z^\star_{i,n}-\hat{z}^\star_{n,\rm AEM}}\nonumber
\end{align}
Hence, we have
\begin{align}
&\abs{z^\star_{i,n}-\hat{z}^\star_{n,\rm AEM}}\leq \frac{2\max\paren{\abs{b},\abs{a}}}{1-L_z}\paren{ \frac{L_z}{n-1}+\frac{1}{n}}\nonumber
\end{align}
which completes the proof.
\section{Proof of Proposition \ref{Lem: CMFA}} \label{App: CMFA}
Note that the $P_n\paren{\alpha}$-based AEM can be written as
 \begin{align}
\hat{z}^\star_{n,\rm AEM}&=\frac{1}{n}\sum_{j=1}^k\sum_{i=1}^n{\rm Br}\paren{\hat{z}^\star_{n,\rm AEM},\alpha_i}\I{\alpha_i\in{C}_j}\nonumber
\end{align}
Thus, we have
\begin{align}
&\abs{\tilde{z}^\star_{n,\rm AEM}-\hat{z}^\star_{n,\rm AEM}}\nonumber\\
&=\abs{\frac{1}{n}\sum_{j=1}^k\sum_{i=1}^n\paren{{\rm Br}\paren{\hat{z}^\star_{n,\rm AEM},\alpha_i}-{\rm Br}\paren{\tilde{z}^\star_{n,\rm AEM},{c}_j}}\I{\alpha_i\in{C}_j}}\nonumber\\
&\leq \frac{1}{n}\sum_{j=1}^k\sum_{i=1}^n\abs{\paren{{\rm Br}\paren{\hat{z}^\star_{n,\rm AEM},\alpha_i}-{\rm Br}\paren{\hat{z}^\star_{n,\rm AEM},{c}_j}}}\I{\alpha_i\in {C}_j}\nonumber\\
&+\frac{1}{n}\sum_{j=1}^k\sum_{i=1}^n\abs{\paren{{\rm Br}\paren{\hat{z}^\star_{n,\rm AEM},{c}_j}-{\rm Br}\paren{\tilde{z}^\star_{n,\rm AEM},{c}_j}}}\I{\alpha_i\in {C}_j}\nonumber\\
&\leq \frac{1}{n}\sum_{j=1}^k\sum_{i=1}^nL\abs{\alpha_i-{c}_j}\I{\alpha_i\in {C}_j}\nonumber\\
&+\frac{1}{n}\sum_{j=1}^k\sum_{i=1}^nL_z\abs{\hat{z}^\star_{n,\rm AEM}-\tilde{z}^\star_{n,\rm AEM}}\I{\alpha_i\in {C}_j}\nonumber\\
&=L_z\abs{\hat{z}^\star_{n,\rm AEM}-\tilde{z}^\star_{n,\rm AEM}}+L\ESI{\hat{\nu}_n}{\abs{\tilde{\nu}_n-\hat{\nu}_n}}\nonumber
\end{align}
\addtolength{\textheight}{-3cm}   
\bibliographystyle{IEEEtran}
\bibliography{IEEEabrv,Meanfield}

\end{document}